\documentclass[a4paper,12pt]{article}
\usepackage[utf8]{inputenc}
\usepackage{amsfonts,amsmath,amssymb,amsthm}
\usepackage[cm]{fullpage}
\newcommand{\Q}{\mathbb Q}
\newcommand{\R}{\mathbb R}
\newcommand{\Z}{\mathbb Z}
\newcommand{\N}{\mathbb N}
\newcommand{\F}{\mathbb F}
\newcommand{\Gal}{\operatorname{Gal}}
\newcommand{\Supp}{\operatorname{Supp}}
\newcommand{\disc}{\operatorname{disc}}

\newtheorem{proposition}{Proposition}
\newtheorem{lemma}{Lemma}
\newtheorem*{theorem}{Theorem}
\newtheorem{theoremm}{Theorem}

\title{Density of the ``quasi $r$--rank Artin problem''}
\author{H.~Abdullah\footnote{Department of Mathematics, College of Sciences, Salahaddin University, Kirkuk Rd. Erbil, IRAQ}, A.~Ali~Mustafa\ ${}^{*,\dag}$ \& F.~Pappalardi\footnote{Dipartimento di Matematica e Fisica, Universit\`a Roma Tre, Largo S. L. Murialdo 1, I--00145, Rome, ITALY}}
\begin{document}

\maketitle

\begin{abstract}
For a given finitely generated multiplicative subgroup of the rationals which possibly contain negative numbers, we derive, subject to GRH, formulas for the densities of primes for which the index of the reduction group has a given value. We completely classify the cases of rank one torsion groups for which the density vanishes and the the set of primes for which the index of the reduction group has a given value, is finite. For higher rank groups we propose some partial results. Finally, we propose some  computations of examples comparing the approximated density computed with 
primes up to $10^{10}$ and that predicted by the Riemann Hypothesis.
\end{abstract}

\section{Introduction}

Let $\Gamma\subset\Q^*$ be a finitely generated multiplicative subgroup. We denote by $\Supp\Gamma$, the \emph{support} of $\Gamma$, i.e.  the finite set of those primes $\ell$ such that the 
$\ell$--adic valuation of some elements of $\Gamma$  is nonzero.

For any prime $p\not\in\Supp\Gamma$, we can define the reduction group:
$$\Gamma_p=\{\gamma\bmod p:\gamma\in\Gamma\}\subset\F_p^*$$
and the prime counting function:
$$\pi_\Gamma(x,m):=\#\{p\le x:p\not\in\Supp\Gamma, [\F_p^*:\Gamma_p]=m\}.$$
We also define the density (if it exists) as
$$\rho(\Gamma,m)=\lim_{x\rightarrow\infty}\frac{\pi_\Gamma(x,m)}{\pi(x)}$$
which exists under the Generalized Riemann Hypothesis and it can be expressed by the following formula (see \cite{Papp}, \cite{CangPapp}, \cite{PS}, \cite{MoreeSteva}):
\begin{equation}\label{rho}
\rho(\Gamma,m)= \sum_{ k \geq 1}
\frac{\mu(k)}{[\Q(\zeta_{mk},\Gamma^{1/mk}):\Q]}.
\end{equation}
Here $\zeta_{d}=e^{2\pi i/d}$ and $\Gamma^{1/d}$ denotes the set of real numbers $\alpha$ such that
$\alpha^d\in\Gamma$.

If $\Gamma=\langle a\rangle$ with $a\in\Q\setminus\{-1,0,1\}$, then the density in question is the density of primes $p$ for which the index of $a$ modulo $p$ equals $m$. In the case $m=1$, the statement that for $a$ not a perfect square, $\rho(\langle a\rangle,1)$ exists and it is not zero, is known as the classical Artin Conjecture for primitive roots which, in 1965, was shown, by C. Hooley \cite{Hoo} to be a consequence of the GRH. Hooley gave a formula for $\rho(\langle a\rangle,1)$ in terms of an euler product which is consistent with (\ref{rho}). 

If we write $a=
\pm b^h$ with $b>0$ not a power of a rational number, $d=\disc(\Q(\sqrt{b}))$ and 
$$F_\ell=\Q(\zeta_{\ell},a^{1/\ell})=\begin{cases}\ell(\ell-1)/\gcd(h,\ell)&\text{if }\ell>2\text{ or }a>0\\
2&\text{if }\ell=2\text{ and }a<0,\end{cases}$$ then:
$$\rho(\langle a\rangle,1)=\left(1-\frac{1-(-1)^{hd}}2
\prod_{\ell\mid d}\frac{-1}{[F_\ell:\Q]-1}
\right)\times\prod_{\ell}\left(1-\frac{1}{[F_\ell:\Q]}\right).$$
In the above and in the sequel, $\ell$ will always denote prime numbers. The case when $m\ge1$ has been considered by various authors \cite{Murata},\cite{Wagstaff}, \cite{Moree}. In particular (Moree \cite[Corollary~2.2]{Moree}), if $m$ is \textbf{odd}, then
\begin{equation}   \label{pap=mor2}
\rho(\langle a\rangle,m)=\left(1-\frac{1-(-1)^{hd}}2
\prod_{\substack{\ell\mid d\\ \ell\nmid 2m}}\frac{-1}{[F_\ell:\Q]-1}\right)\frac{(m,h)}{m^2}\times\prod_{\substack{\ell\mid m\\ h_\ell\mid m_\ell}}\left(1+\frac1\ell\right)\times
 \prod_{\ell\nmid m}\left(1-\frac{1}{[F_\ell:\Q]}\right). 
\end{equation}
 In the above and in the sequel, $m_\ell$ will always denote the $\ell$--part of $m$ (i.e. $m_\ell=\ell^{v_\ell(m)}$ where $v_\ell$ is the $\ell$--adic valuation).
A formula for the remaining case, $m$ \textbf{even}, can be found in \cite[Theorem~2.2]{Moree}.


The case when the rank of $\Gamma$ is greater than 1 was considered in \cite{CangPapp,MoreeSteva,PS,CP1}. For $\Gamma\subset\Q^*$ finitely generated subgroup and $m\in\N$, we set $\Gamma(m):=\Gamma\cdot{\Q^*}^m/{\Q^*}^m$ and
\begin{equation}\label{AGamma}
 A(\Gamma,m)=\frac1{\varphi(m)|\Gamma(m)|}
\times\prod_{\substack{\ell>2\\ \ell\mid
m}}\left(1-\frac{|\Gamma(m_\ell)|}{\ell|\Gamma(\ell m_\ell)|}\right)\times
\prod_{\substack{\ell>2\\ \ell\nmid
m}}\left(1-\frac1{(\ell-1)|\Gamma(\ell)|}\right).
\end{equation}
For $\gamma\in\Gamma(m)$, $\gamma'\in\Z$ denotes the unique, up to sign, $m$--power free representative of $\gamma$ ($\gamma=\gamma'\cdot{\Q^*}^m$). 
The sign of $\gamma'$ is chosen to be positive if $m$ is odd or if $\gamma=\gamma'\cdot{\Q^*}^m\subset\Q^+$ and is negative otherwise.  
If $\alpha>0$ and $\gamma
\in\Gamma(2^\alpha)[2]$ (the 2-torsion subgroup of $\Gamma(2^\alpha)$) with $\gamma\neq{\Q^*}^{2^\alpha}$, then  $\gamma'=\pm\gamma_0^{2^{\alpha-1}}$ with $\gamma_0\in\N, \gamma_0>1$ square free. We shall denote by 
$\delta(\gamma)=\disc\Q(\sqrt{\gamma_0})$ which is easily seen to depend only on $\gamma$.

For $\Gamma\subset\Q^+$, we define the group:
\begin{equation}\label{tildegammapositive0}
 \tilde{\Gamma}(m)
 :=
 \left\{\gamma\in\Gamma(m_2)[2]: v_2(\delta(\gamma))\le v_2(m)\right\}.
\end{equation}
It is easy to check that $\tilde{\Gamma}(m)$ is a  $2$--group. 
If $\Gamma\subset\Q^+$, then
\begin{equation}\label{tildegammapositive}
 \tilde{\Gamma}(m)=
 \begin{cases} \{1\} &\text{if }2\nmid m\\
\left\{\gamma\in\Gamma(2): \gamma'\equiv1\bmod4\right\}& \text{if } 2\| m\\                    
\left\{\gamma\in\Gamma(4)[2]: 2\nmid \gamma_0\right\}& \text{if } 4\| m\\                    
\Gamma(m_2)[2]& \text{if } 8\| m.                                     \end{cases}
\end{equation}
The group $\tilde{\Gamma}(m)$ will be defined also in the case when
$\Gamma\nsubseteq\Q^+$ in (\ref{tildegamma}).

Finally, we set:
\begin{equation}
 \label{BGamma}
B_{\Gamma,k} =\!\!\!
\sum_{\gamma\in\tilde{\Gamma}(k)}
\prod_{\substack{\ell\mid \delta(\gamma)\\
\ell\nmid k}}\frac{-1}{(\ell-1)|\Gamma(\ell)|-1}.
\end{equation}

For the special case when $\Gamma$ contains only positive rational numbers, in \cite{PS}, it was proved the following:

\begin{theorem}\label{quasi_r_artin_density+} Let $\Gamma \subset \Q^+:=\{q\in\Q: q>0\}$ be multiplicative subgroup of rank $r$
and let $m\in\N$. Then
$$\rho(\Gamma,m)=A(\Gamma,m)\times\left(  B_{\Gamma,m}
-\frac{|\Gamma(m_2)|}{(2,m)|\Gamma(2m_2)|} B_{\Gamma,2m}\right) $$
where $A(\Gamma,m)$ is defined in (\ref{AGamma}) and  $B_{\Gamma,k}$ is defined as in (\ref{BGamma}).
\end{theorem}

Note that, for $m$ odd,  $B_{\Gamma,m}=1$  and the formula above specializes to 
\begin{eqnarray}\nonumber
 \rho(\Gamma,m)&=&\frac1{\varphi(m)|\Gamma(m)|}
\prod_{\ell\mid m}\left(1-\frac{|\Gamma(m_\ell)|}{\ell|\Gamma(\ell m_\ell)|}\right)\prod_{\ell\nmid
m}\left(1-\frac1{(\ell-1)|\Gamma(\ell)|}\right)\times\\ &&\times
\left(  1
+\sum_{\substack{\gamma\in\Gamma(2)\setminus\{{\Q^*}^2\}\\ \delta(\gamma)\equiv1\bmod4}} 
\prod_{\substack{\ell\mid 2\delta(\gamma)\\
\ell\nmid m}}\frac{-1}{(\ell-1)|\Gamma(\ell)|-1}
\right)\label{modd}\end{eqnarray}
which, for $m=1$,  should be compared with \cite[4.6.~Theorem]{MoreeSteva}.
Furthermore,
one can check that the formula in the above result from \cite{PS} coincides with that of Moree's \cite[Theorem~2.2]{Moree}  in the case when $\Gamma=\langle a\rangle$ with $a\in\Q^+, a\ne1$.

The goal of this paper is to extend the above Theorem by removing the constraint that $\Gamma \subset \Q^+$. We prove the following:

\begin{theoremm}\label{quasi_r_artin_density} Let $\Gamma \subset \Q^*$ be multiplicative subgroup of rank $r\ge1$
and let $m\in\N$. Let
\begin{equation}\label{tildegamma}
\tilde{\Gamma}(m)=\left\{\gamma\in\Gamma(m_2)[2]: \begin{array}{l}
\text{if }\gamma\subset\Q^+\text{ then }v_2(\delta(\gamma))\le v_2(m);\\\text{if }\gamma\nsubseteq\Q^+\text{ then }v_2(\delta(\gamma))=v_2(m)+1\end{array}\right\}.
\end{equation}
Then, with $A(\Gamma,m)$ defined as in (\ref{AGamma}) and $B_{\Gamma,k}$ defined as in (\ref{BGamma}),
$$\rho(\Gamma,m) = A(\Gamma,m)\left( B_{\Gamma,m}
-\frac{|\Gamma(m_2)|}{(2,m)|\Gamma(2m_2)|}
B_{\Gamma,2m}\right).$$
\end{theoremm}
Clearly, the definition of $\tilde{\Gamma}(m)$ in Theorem~\ref{quasi_r_artin_density} reduces to the one in (\ref{tildegammapositive0}) when $\Gamma\subset\Q^+$.
 Furthermore,  it is not hard to verify that:  
\begin{equation}\label{tildegammam}\tilde\Gamma(m)=
\begin{cases}
 \{1\} &\text{if }2\nmid m;\\
 \{\gamma\in\Gamma(2):  \gamma'\equiv1\bmod4\}&\text{if }2\|m;\\
  \{\gamma\in\Gamma(4): \text{either }\gamma'=\gamma_0^2, 
    2\nmid \gamma_0\text{ or }\gamma'=-\gamma_0^2, 2\mid \gamma_0\}
    &\text{if }4\|m;\\
  \Gamma(m_2)[2]\cap\Q^+&\text{if }8\mid m.
\end{cases}
\end{equation}
Hence $\tilde\Gamma(m)$ is also a $2$--group. The above identity should be compared with (\ref{tildegammapositive}). If $m$ is odd, then the formula for $ \rho(\Gamma,m)$ in the statement of Theorem~\ref{quasi_r_artin_density}
simplifies  to the same as in (\ref{modd}).

In Section~\ref{1a} we specialize to the case when $\Gamma=\langle -1,a\rangle$ where $a\in\Q^*\setminus\{0,1,-1\}$ can be assumed to be positive. We deduce from Theorem~\ref{quasi_r_artin_density} an explicit formulas for $\rho_{\langle-1,a\rangle,m}$ which is used in Section~\ref{-1,a} to prove the following:
\begin{theoremm}\label{week_vanish} Let $a\in\Q^+\setminus\{-1,0,1\}$, write $a=a_0^h$, where $a_0\in\Q^+$ 
not the power of any rational number and write $a_0=a_1a_2^2$ where $a_1>1$ is uniquely defined by the property to be a positive square free integer.
 The density $\rho(\langle -1,a\rangle,m)=0$ if and only if  one of the
following two (mutually exclusive) cases is verified:
\begin{enumerate}
 \item $3\mid h,  3\nmid m, 3\mid  a_1,  a_1\mid3m,\quad
2\nmid h, 2\|m, 2\nmid a_1,$;
 \item $3\mid h,  3\nmid m, 3\mid  a_1,  a_1\mid3m,\quad  v_2(h)<v_2(m)\ne1$.
 \end{enumerate}
Furthermore, on GRH, the set $\{p: \left[\F_p^*:\langle-1,a\rangle_p\right]=m\}$ is finite if and only if one of the above two conditions is satisfied. 
\end{theoremm}

Examples of pairs $(a,m)$ satisfying 1. of Theorem~\ref{week_vanish} are $(a,m)=(3^3,2), (15^3,10),\cdots$ and examples
of pairs satisfying 2. are $(a,m)=(3^6,8),(15^{12},40),\cdots$. A list of more values of $(a,m)$ is presented in the second table of Section~\ref{computations}.


Next, in Section~\ref{vanishing}, we investigate the identity
$$\rho(\Gamma,m)=0$$
and
the problem of determining whether $$\mathcal{N}_{\Gamma,m}=\{
p\not\in\Supp\Gamma, \operatorname{ind}_p\Gamma=m\}$$ is finite. 

If $\Gamma=\langle g\rangle$ with 
$g\in\Q\setminus\{0,1,-1\}$, this problem has been solved (on GRH) by Lenstra 
\cite[(8.9)--(8.13)]{Lenstra} (see also \cite{Moree}).
In fact, 
\begin{theorem}{Lenstra \cite[Theorem~4]{Moree}} Let $g\in\Q\setminus\{-1,0,1\}$ and write $g=\pm g_0^h$, where $g_0\in\Q^+$ is 
not the power of any rational number. 
 The density $\rho(\langle g\rangle,m)=0$ if and only if we are in one of the
following six (mutually exclusive) cases:
\begin{enumerate}
 \item $2\nmid m$, $\operatorname{disc}(\Q(\sqrt g))\mid m$;
 \item $g> 0$,  $v_2(m)>v_2(h)$, $3\mid h$, $3\nmid m$, $\operatorname{disc}(\Q(\sqrt{-3g_0}))\mid m$;
 \item $g<0$, $2\nmid h$, $2\| m$, $3\nmid m$, $3\mid h$, $\operatorname{disc}(\Q(\sqrt{3g_0}))\mid m$;
 \item $g<0$, $2\| h$, $2\| m$, $\operatorname{disc}(\Q(\sqrt{2g_0}))\mid 2m$;
 \item $g<0$, $2\| h$, $4\| m$, $3\mid h$, $3\nmid m$, $\operatorname{disc}(\Q(\sqrt{-6g_0}))\mid m$;
 \item $g<0$, $v_2(m)>1+v_2(h)$, $3\mid h$, $3\nmid m$, $\operatorname{disc}(\Q(\sqrt{-3g_0}))\mid m$.
\end{enumerate}
Furthermore, on GRH, $\mathcal{N}_{\langle g\rangle,m}$ is finite if and only if one of the above two conditions
is satisfied. 
\end{theorem}

In the higher rank case, we partially generalize the above in the following way:
\begin{theoremm}\label{finite} Let $\Gamma\subset\Q^+$ be a non--trivial, finitely generated subgroup and let $m\in\N$. Then
$\rho(\Gamma,m)=0$ when one of the following three conditions is satisfied:
\begin{enumerate}
 \item[A.] $2\nmid m$ and 
for all $g\in\Gamma, \disc(\Q(\sqrt g))\mid m$;
 \item[B.] $2\mid m$, $3\nmid m$, $\Gamma(3)$ is trivial and there exists 
$\gamma_1\in\tilde\Gamma(m)$ such that $3\mid\delta(\gamma_1)\mid6m$. 
\item[C.] $2\| m$, $|\Gamma(2)|=2, \tilde{\Gamma}(2m)=\Gamma(4)$ and 
for all $\gamma\in\tilde{\Gamma}(2m)$, $\delta(\gamma)\mid 4m$.
\end{enumerate}
\end{theoremm}

\noindent\textsc{Remark}. Regarding the last property of Theorem~\ref{finite}, note that   $\Gamma$ and $m$ satisfy $2\| m$, then
$$|\Gamma(2)|=2\qquad\text{and}\qquad\tilde{\Gamma}(2m)=\Gamma(4)$$
if and only if 
\begin{enumerate}
 \item $\Gamma(2)=\{{\Q^*}^2,-{\Q^*}^2\}$;
 \item the elements of $\Gamma(4)$ are of the form $\gamma_0^2{\Q^*}^4$ or $-4\gamma_0^2{\Q^*}^4$ with $\gamma_0\in\N$ odd and square free;
 \item $\Gamma(4)$ contains at least one element on the second form.
\end{enumerate}
 In fact, if $g{\Q^*}^2\in\Gamma(2)$ with $g\in\Gamma$ and $|g|$ not a perfect square, then $g{\Q^*}^4\in\Gamma(4)$ is an element of order $4$ so that $\tilde{\Gamma}(2m)$ is proper subgroup of $\Gamma(4)$.
The form of $\tilde\Gamma(m)$ is described in (\ref{tildegammam}). Finally, at least one of the elements has to be of the form  $-4\gamma_0^2{\Q^*}^4$, otherwise $\Gamma(2)=\{{\Q^*}^2\}$.

The result in Theorem~\ref{finite} is compatible with the result of Lenstra. In fact

\begin{proposition}\label{lenstraok} Suppose $\Gamma=\langle g\rangle$ and $m\in\N$. Then
condition A. of Theorem~\ref{finite} reduces to condition 1. of Lenstra's Theorem, condition C. reduces
to condition 4 and condition B. reduced to one of conditions 2, 3, 5 or 6 according to the following:
\begin{center}
\begin{tabular}{|l|l|}
\hline
2.& if $g>0$\\
3.& if $g<0$, $v_2(m)=1$ and $v_2(h)=0$ \\
5.& if $g<0$, $v_2(m)=2$ and $v_2(h)=1$\\
6.& if $g<0$ and $v_2(m)>v_2(h)+1$\\
\hline
\end{tabular}
\end{center}
where $g=\pm g_0^h$ with $g_0\ne1$ not the power of a rational number.
\end{proposition}

When $\operatorname{rank}\Gamma>2$, we do not know in general if $\rho(\Gamma,m)=0$ implies that  at least one of the conditions of Theorem~\ref{finite} is satisfied.  Possibly the approach due to Lenstra, Moree and Stevenhagen \cite{LMS} could provide a complete characterization of the pairs $\Gamma, m$ with $\rho(\Gamma,m)=0$
also in the case when $\Gamma$ contains some negative rational numbers. The techniques of \cite{LMS}
have been adapted to the context of higher rank groups by Moree and Stevenhagen in \cite{MoreeSteva}
where the case $m=1$ is considered.  On the other hand, a least in the case when $m$ is odd, condition 1. of Theorem~\ref{finite} is also necessary. In fact we have the following:

\begin{proposition}\label{oddsuff} Assume that $2\nmid m$ and $\rho(\Gamma,m)=0$. Then condition 1. of Theorem~\ref{finite} is satisfied.
\end{proposition}

We conclude with the following:
\begin{proposition}\label{ultima} Assume that $\Gamma\subset\Q^*$ and $m$ satisfy one of the three conditions of Theorem~\ref{finite}, then
$\mathcal{N}_{\Gamma,m}$ is finite. Hence, on GRH, if $2\nmid m$,
$$\mathcal{N}_{\Gamma,m}\text{ finite } \Longleftrightarrow \forall \gamma{\Q^*}^2\in\Gamma(2), \disc(\Q(\sqrt\gamma))\mid m.$$
\end{proposition}

\section{The degree of Kummer extensions}
In this section we are interested on determining an explicit formula
for the order of the Galois group $\#\Gal(\Q(\zeta_m,\Gamma^{1/d})/\Q)=[\Q(\zeta_m,\Gamma^{1/d}):\Q]$ where $d\mid m$, $\zeta_m=e^{2\pi i/m}$ and 
$\Gamma^{1/d}=\{\sqrt[d]{\alpha}\in\R: \alpha\in\Gamma\}$.

By the standard properties of Kummer extensions (see for example \cite[Theorem~8.1]{Lang}), if we denote by $K_m=\Q(\zeta_m)$ the cyclotomic field,  we have that
\begin{equation}\label{2020_0}\Gal(K_m(\Gamma^{1/d})/K_m)\cong\Gamma(d)/\tilde{\Gamma}_{m,d}
 \end{equation}

where $\Gamma(d):=\Gamma\cdot{\Q^*}^d/{\Q^*}^d$ and $\tilde{\Gamma}_{m,d}:=\left(\Gamma\cdot{\Q^*}^d\cap {K_m^*}^d\right)/{\Q^*}^d$. Note that if $d>1$ is odd, then ${K_m^*}^d\cap\Q^*={\Q^*}^d$,
Hence
$$\tilde{\Gamma}_{m,d}\cong\prod_{\ell\mid d}\tilde{\Gamma}_{m,d_\ell}=\tilde{\Gamma}_{m,d_2}.$$

We recall that for $\gamma\in\Gamma(d)$, $\gamma'\in\Z$ denotes the unique, up to sign, $d$--power free representative of $\gamma$ ($\gamma=\gamma'\cdot{\Q^*}^d$). 
The sign of $\gamma'$ is chosen to be positive if $d$ is odd or if $\gamma=\gamma'\cdot{\Q^*}^d\subset\Q^+$ and is negative otherwise.  
 Therefore
\begin{equation}\label{2020_1}
\tilde{\Gamma}_{m,2^\alpha}=\{\gamma\in\Gamma(2^\alpha):\
\gamma'\in\Gamma\cdot{\Q^*}^{2^\alpha}\cap{K_m^*}^{2^\alpha}\}. 
\end{equation}

It was observed in \cite[Corollary~1]{Papp} that,  for $2^\alpha\mid m$,
\begin{equation}\label{vecchia_uno}
\text{if }\Gamma\subset\Q^+\qquad\text{then}\qquad
{\tilde{\Gamma}}_{m,2^\alpha} =\{
\gamma\in\Gamma(2^\alpha)[2]: \delta(\gamma)\mid m\}. 
\end{equation}
In fact, if $\Gamma\subset\Q^+$ and $\gamma'\in\Gamma(2^\alpha)[2]$, then $\gamma'=\gamma_0^{2^{\alpha-1}}$
and $\delta(\gamma)=\disc\Q(\sqrt{\gamma_0}).$

Furthermore, if $\alpha=0$, then ${\tilde{\Gamma}}_{m,1}$ is the trivial group and  in \cite[page~124, (24)]{CangPapp} is was proven that
if $\alpha=1$ then, 
\begin{eqnarray}\label{2020_2}\text{if }m\text{ is squarefree}\quad\text{then}\quad {\tilde{\Gamma}}_{m,2}=\{\gamma\in\Gamma(2): \disc\Q(\sqrt{\gamma'})\mid m\text{ and }\disc\Q(\sqrt{\gamma'})\equiv1(\bmod\ 4)\}.
\end{eqnarray}

Note that for $4\nmid m$, the condition $\disc\Q(\sqrt{\gamma'})\equiv1(\bmod\ 4)$ above is irrelevant as it is implied
by the condition that $\disc\Q(\sqrt{\gamma'})\mid m$. Hence, for $m$ square free, the formula in  (\ref{2020_2}) and that in (\ref{vecchia_uno}) coincide.

Our first task is to extend the above formula for ${\tilde{\Gamma}}_{m,2}$ in the case when $m$ is not necessarily
squarefree.
\begin{proposition}\label{2020_3} Let $\Gamma\subset\Q^*$ be a finitely generated subgroup, let $m\in\N$ be even. Then
$${\tilde{\Gamma}}_{m,2}=\{\gamma\in\Gamma(2):
 \disc\Q(\sqrt{\gamma'})\mid m\}$$
\end{proposition}
Although the proof of the Proposition is the same as the proof of Corollary~1 in \cite{Papp}, we add it here for completeness.

\begin{proof}[Proof of the Proposition] Let us start from the definition if
(\ref{2020_1}):
$$\tilde{\Gamma}_{m,2}:=\{\gamma\in\Gamma(2):\
\gamma'\in\Gamma\cdot{\Q^*}^2\cap{K_m^*}^2\},$$
where $K_m=\Q(\zeta_m)$. 
If $\gamma'\in\Gamma\cdot{\Q^*}^2$ is a squarefree integer, then
$\gamma'\in{K_m^*}^{2}$ if and only if $\sqrt{\gamma'}\in{K_m}^*$
and this happens if and only if $\operatorname{disc}\Q(\sqrt{\gamma'})\mid m$ (see for example Weiss \cite[page 264]{We}).
This completes the proof.
\end{proof}

We have the general 

\begin{lemma}\label{stura}
 Let $\Gamma\subset\Q^*$ be a finitely generated group. Let $m\in\N$ and let $\alpha\in\N$, $\alpha\ne0$ be such that $2^\alpha\mid m$. Finally set
 $$\tilde{\Gamma}_{m,2^\alpha}^+=\{\gamma\in\Gamma(2^\alpha)[2]: 
\gamma\subset\Q^+, \delta(\gamma)\mid m\}$$
and
$$\tilde{\Gamma}_{m,2^\alpha}^-=\begin{cases}
                              \{\gamma\in\Gamma(2^\alpha)[2]: 
\gamma\nsubseteq\Q^+, \delta(\gamma)\mid m\}&\text{if }2^{\alpha+1}\mid m\\
\{\gamma\in\Gamma(2^\alpha)[2]: 
\gamma\nsubseteq\Q^+,\delta(\gamma)\mid 2m\text{ but }\delta(\gamma)\nmid m\}&\text{if }2^{\alpha}\| m
                             \end{cases}
$$
where, if $\gamma'=\pm\gamma_0^{2^{\alpha-1}}$, $\delta(\gamma):=\disc(\Q(\sqrt{\gamma_0})$. Then
$$\tilde{\Gamma}_{m,2^\alpha}=\tilde{\Gamma}_{m,2^\alpha}^+\cup\tilde{\Gamma}_{m,2^\alpha}^-.$$
\end{lemma}

The proof is, in spirit, the same as the proof of \cite[Lemma~4]{Schinzel}.

\begin{proof} We start from the definition:
\begin{equation*}
\tilde{\Gamma}_{m,2^\alpha}=\{\gamma\in\Gamma(2^\alpha):\
\gamma'\in\Gamma\cdot{\Q^*}^{2^\alpha}\cap{K_m^*}^{2^\alpha}\}. 
\end{equation*}

Suppose first $\gamma=\gamma'{\Q^*}^{2^\alpha}\subset\Q^+$ with 
 $\gamma'\in\N$, $2^\alpha$--power free and that $\sqrt[2^\alpha]{\gamma'}\in\Q(\zeta_m)$. Then 
 $\Q(\sqrt[2^\alpha]{\gamma'})$ is a Galois, real, extension of $\Q$ and this can only happen if its degree over $\Q$ is at most
 $2$. Hence $\gamma'=\gamma_0^{2^{\alpha-1}}$ for some square free $\gamma_0\in\N$ so that $\delta(\gamma)=\operatorname{disc}\Q(\sqrt{\gamma_0})\mid m$ and $\gamma\in\Gamma(2^\alpha)[2]$. Hence
 $\gamma\in\tilde{\Gamma}_{m,2^\alpha}^+$. 
 
 Next suppose that $\gamma=\gamma'{\Q^*}^{2^\alpha}\nsubseteq\Q^+$, $\gamma'\in\Z$ and $\gamma'<0$. The condition $\gamma'\in {K_m^*}^{2^\alpha}$ implies that $\gamma'^2\in{K_m^*}^{2^{\alpha+1}}$ is positive. Therefore, by the argument above, $\gamma'^2=\gamma_0^{2^{\alpha}}$ for some square free  $\gamma_0\in\N$. Finally $\gamma'=-\gamma_0^{2^{\alpha-1}}\in {K_m^*}^{2^\alpha}$.
 
 From this property we deduce that $$\sqrt[2^\alpha]{\gamma'}=\varepsilon\sqrt{\gamma_0}\in K_m^*$$
 for some primitive $2^{\alpha+1}$--root of unity $\varepsilon$. We need to distinguish two cases: $2^{\alpha+1}\mid m$
or $2^\alpha\|m$. 

If $2^{\alpha+1}\mid m$, $\varepsilon\in K_m$. So $\sqrt{\gamma_0}\in K_m$  which is equivalent to $\delta(\gamma)\mid m$.

If $2^{\alpha}\| m$, $\varepsilon\in K_{2m}\setminus K_m$. $\sqrt{\gamma_0}\in K_{2m}\setminus K_m$ which is equivalent to $\delta(\gamma)\mid 2m$
but $\delta(\gamma)\nmid m$.

This discussion proves that 
$$\tilde{\Gamma}_{m,2^\alpha}\subseteq\tilde{\Gamma}_{m,2^\alpha}^+\cup\tilde{\Gamma}_{m,2^\alpha}^-.$$
 Viceversa, suppose that $\gamma\in\tilde{\Gamma}_{m,2^\alpha}^+\cup\tilde{\Gamma}_{m,2^\alpha}^-$ and that $\gamma\ne{\Q^*}^{2^\alpha}$. Then $\gamma=\pm\gamma_0^{2^{\alpha-1}}{\Q^*}^{2^\alpha}$ and the condition $\delta(\gamma)=\operatorname{disc}\Q(\sqrt{\gamma_0})\mid m$ is equivalent to $\sqrt{\gamma_0}\in K_m$.
 
Finally, if $\gamma\in\tilde{\Gamma}_{m,2^\alpha}^+$, $\gamma'=\gamma_0^{2^{\alpha-1}}=\left(\sqrt{\gamma_0}\right)^{2^\alpha}\in {K_m^*}^{2^\alpha}$
and hence $\gamma'\in {\Q^*}^{2^\alpha}\cap{K_m^*}^{2^\alpha}$ 
so that $\tilde{\Gamma}_{m,2^\alpha}^+\subset\tilde{\Gamma}_{m,2^\alpha}$,
while
if $\gamma\in\tilde{\Gamma}_{m,2^\alpha}^-$, $\gamma'=-\gamma_0^{2^{\alpha-1}}=\left(\varepsilon\sqrt{\gamma_0}\right)^{2^\alpha}$, for some primitive $2^{\alpha+1}$--root of unity $\varepsilon$.

If $2^{\alpha+1}\mid m$, then $\varepsilon\in K_m^*$ and hence $\gamma'\in\Gamma\cdot {\Q^*}^{2^\alpha}\cap{K_m^*}^{2^\alpha}$ 
so that $\tilde{\Gamma}_{m,2^\alpha}^-\subset\tilde{\Gamma}_{m,2^\alpha}$.

Suppose $2^{\alpha}\| m$. If $\gamma\in\tilde{\Gamma}_{m,2^\alpha}^-$, 
then $\gamma'=-\gamma_0^{2^{\alpha-1}}$ and $\gamma'^2=\gamma_0^{2^{\alpha}}=(\sqrt{-\gamma_0})^{2^{\alpha+1}}\in K_m^{2^{\alpha+1}}$ since the condition
$\delta(\gamma)\mid 2m$ but 
$\delta(\gamma)\nmid m$ implies that $\sqrt{-\gamma_0}\in K_m^*$. 
Therefore either $\gamma'\in {K_m^*}^{2^\alpha}$ or $-\gamma'\in {K_m^*}^{2^\alpha}$ . If it was that $-\gamma'= \gamma_0^{2^{\alpha-1}}\in {K_m^*}^{2^\alpha}$ we would deduce that $\sqrt{\gamma_0}\in K_m^*$ and this would contradic $\delta(\gamma)\nmid m$. Finally  $\gamma'\in\Gamma\cdot {\Q^*}^{2^\alpha}\cap{K_m^*}^{2^\alpha}$ 
so that $\tilde{\Gamma}_{m,2^\alpha}^-\subset\tilde{\Gamma}_{m,2^\alpha}$.
\end{proof}

\noindent\textsc{Remark.}
Let $\gamma_0\in\N$ be square free and suppose that $2\| m$. Then the condition  $\disc(\Q(\sqrt{-\gamma_0}))\mid m$ is equivalent to 
$\disc(\Q(\sqrt{\gamma_0}))\mid 2m$ and $\disc(\Q(\sqrt{\gamma_0}))\nmid m$. In fact
with the given assumption on $\gamma_0$ and $m$, $\disc(\Q(\sqrt{-\gamma_0}))\mid m$ if and only if $\gamma_0\equiv 3\bmod 4$ and $\gamma_0\mid m/2$ so that
$\disc(\Q(\sqrt{\gamma_0}))=4\gamma_0\mid 2m$ and $\disc(\Q(\sqrt{\gamma_0}))=4\gamma_0\nmid m$.
This explains why the formula in Lemma~\ref{stura} reduces to the one in
Proposition~\ref{2020_3} in the case when $\alpha=1$.\medskip

\section{Proof of Theorem~\ref{quasi_r_artin_density}}

Let us start by writing $m=2^{{v_2(m)}}n$ with $2\nmid n$ and note that
\begin{eqnarray*}
\rho(\Gamma,m)&=& \sum_{ k \geq 1}
\frac{\mu(k)}{[\Q(\zeta_{mk},\Gamma^{1/mk}):\Q]}=
\sum_{ k \geq 1}
\frac{\mu(k)\left|\tilde{\Gamma}_{mk,m_2k_2}\right|}{\varphi(mk)\left|\Gamma(mk)\right|}
\\ &=& \sum_{\alpha=0}^\infty
\frac{\mu(2^\alpha)}{\varphi(2^{\alpha+{v_2(m)}})\left|\Gamma(2^{\alpha+{v_2(m)}})\right|}
\sum_{\substack{k \geq 1\\ 2\nmid k}}
\frac{\mu(k)\left|\tilde{\Gamma}_{2^{\alpha+{v_2(m)}}nk,2^{\alpha+{v_2(m)}}}\right|}
{\varphi(nk)\left|\Gamma(nk)\right|}\\ &=&\!\!\!\!\!\!
\sum_{\alpha={v_2(m)}}^\infty
\frac{\mu(2^{\alpha-{v_2(m)}})}{\varphi(2^{\alpha})\left|\Gamma(2^{\alpha})\right|}
\left(
\sum_{\substack{k \geq 1\\ 2\nmid k}}
\frac{\mu(k)\left|\tilde{\Gamma}^+_{2^{\alpha}nk,2^{\alpha}}\right|}
{\varphi(nk)\left|\Gamma(nk)\right|}
+
\sum_{\substack{k \geq 1\\ 2\nmid k}}
\frac{\mu(k)\left|\tilde{\Gamma}^-_{2^{\alpha}nk,2^{\alpha}}\right|}
{\varphi(nk)\left|\Gamma(nk)\right|}
\right)
\\ &=&\!\!\!\!\!\!
\sum_{\alpha={v_2(m)}}^\infty
\frac{\mu(2^{\alpha-{v_2(m)}})}{\varphi(2^{\alpha})\left|\Gamma(2^{\alpha})\right|}
\left(
\sum_{\substack{\gamma\in\Gamma(2^\alpha)[2]\\ \gamma\subset\Q^+}}
\sum_{\substack{k \geq 1, 2\nmid k\\ \delta(\gamma)\mid 2^\alpha kn}}
\frac{\mu(k)}
{\varphi(nk)\left|\Gamma(nk)\right|}
+\!\!\!\!
\sum_{\substack{\gamma\in\Gamma(2^\alpha)[2]\\ \gamma\nsubseteq\Q^+}}
\sum_{\substack{k \geq 1, 2\nmid k\\ \delta(\gamma)\mid 2^{1+\alpha} kn\\ \delta(\gamma)\nmid 2^{\alpha} kn}}
\frac{\mu(k)}
{\varphi(nk)\left|\Gamma(nk)\right|}
\right).
\end{eqnarray*}

\begin{lemma}\label{tecn} Suppose that $\delta$ is a squarefree odd integer, that $n$ is an odd integer and set: 
$$A_{\Gamma,n}=\frac1{\varphi(n)|\Gamma(n)|}
\times\prod_{\substack{\ell>2\\ \ell\mid
n}}\left(1-\frac{|\Gamma(n_\ell)|}{\ell|\Gamma(\ell n_\ell)|}\right)\times
\prod_{\substack{\ell>2\\ \ell\nmid
n}}\left(1-\frac1{(\ell-1)|\Gamma(\ell)|}\right).$$
Then the following identity holds
$$\sum_{\substack{k \geq 1, 2\nmid k\\ \delta\mid kn}}
\frac{\mu(k)}
{\varphi(nk)\left|\Gamma(nk)\right|}= A_{\Gamma,n}\prod_{\substack{\ell\mid\delta\\ \ell\nmid n}}
\frac{-1}{(\ell-1)|\Gamma(\ell)|-1}.$$
\end{lemma}

\begin{proof}
Observe that $\delta\mid kn$ if and only if $d:=\delta/\gcd(\delta,n)\mid k$. If we write
$k=dt$, then $\gcd(d,n)=\gcd(d,t)=1$, so that $\varphi(ndt)|\Gamma(ndt)|=\varphi(d)|\Gamma(d)|\times\varphi(nt)|\Gamma(nt)|$ and
\begin{eqnarray*}
 \sum_{\substack{k \geq 1, 2\nmid k\\ \delta\mid kn}}
\frac{\mu(k)}
{\varphi(nk)\left|\Gamma(nk)\right|}&=&
\sum_{\substack{t \geq 1, \\ \gcd(t,2d)=1}}\frac{\mu(dk)}{\varphi(ndt)|\Gamma(ndt)|}\\
&=&\frac{1}{\varphi(n)|\Gamma(n)|}\times\frac{\mu(d)}{\varphi(d)|\Gamma(d)|}\times\sum_{\substack{t \geq 1, \\ \gcd(t,2d)=1}}\frac{\mu(t)\varphi(t,n)|\Gamma(n)|}{\gcd(n,t)\varphi(t)|\Gamma(nt)|}
\end{eqnarray*}
where we used the identity $\varphi(tn)=\varphi(t)\varphi(n)\gcd(t,n)/\varphi(\gcd(t,n))$. Since $\frac{|\Gamma(n)|}{|\Gamma(nt)|}$ is a multiplicative function of $t$, the above equals:
\begin{eqnarray*}
&=&
\frac{1}{\varphi(n)|\Gamma(n)|}\times\frac{\mu(d)}{\varphi(d)|\Gamma(d)|}\times
\prod_{\ell\nmid 2d}\left(1-\frac{\varphi(\gcd(\ell,n))|\Gamma(n_\ell)|}{\varphi(\ell)\gcd(n,\ell)|\Gamma(\ell n_\ell)|}\right)\\
&=&
 \frac{1}{\varphi(n)|\Gamma(n)|}\times\frac{\mu(d)}{\varphi(d)|\Gamma(d)|}\times
\prod_{\ell\nmid 2d,\ell\mid n}\left(1-\frac{|\Gamma(n_\ell)|}{\ell|\Gamma(\ell n_\ell)|}\right)
\times
\prod_{\ell\nmid 2dn}
\left(1-\frac{1}{(\ell-1)|\Gamma(\ell n_\ell)|}\right)\\
&=&A_{\Gamma,n}\times\frac{\mu(d)}{\varphi(d)|\Gamma(d)|}\times
\prod_{\ell\mid d,\ell>2}
\left(1-\frac{1}{(\ell-1)|\Gamma(\ell)|}\right)^{-1}\\
&=& A_{\Gamma,n}\times\prod_{\substack{\ell\mid\delta\\ \ell\nmid 2n}}
\frac{-1}{(\ell-1)|\Gamma(\ell)|-1}.
\end{eqnarray*}
\end{proof}

From Lemma~\ref{tecn}, we deduce that 
\begin{eqnarray*}
\rho(\Gamma,m)&=& A_{\Gamma,n}
\sum_{{v_2(m)}\le\alpha\le {v_2(m)}+1}
\frac{\mu(2^{\alpha-{v_2(m)}})}{\varphi(2^{\alpha})\left|\Gamma(2^{\alpha})\right|}\times\\ &&\hfill \times
\left(
\sum_{\substack{\gamma\in\Gamma(2^\alpha)[2]\\ \gamma\subset\Q^+\\
v_2(\delta(\gamma))\le\alpha}}
\prod_{\substack{\ell\mid\delta(\gamma)\\ \ell\nmid 2n}}
\frac{-1}{(\ell-1)|\Gamma(\ell)|-1}
+\!\!\!\!\!\!\!\!
\sum_{\substack{\gamma\in\Gamma(2^\alpha)[2]\\ \gamma\nsubseteq\Q^+\\ v_2(\delta(\gamma))=\alpha+1}}\!\!
\prod_{\substack{\ell\mid\delta(\gamma)\\ \ell\nmid 2n}}
\frac{-1}{(\ell-1)|\Gamma(\ell)|-1}
\right)\\
&=& A_{\Gamma,n}\times\left( B_{\Gamma,m}
-\frac{|\Gamma(m_2)|}{(2,m)|\Gamma(2m_2)|}
B_{\Gamma,2m}\right)
\end{eqnarray*}
where
\begin{equation}\label{Bgamma}B_{\Gamma,m}=\sum_{\gamma\in\tilde\Gamma(m)}\prod_{\substack{ \ell\mid\delta(\gamma)\\ \ell\nmid 2m}}
\frac{-1}{(\ell-1)|\Gamma(\ell)|-1}
\end{equation}
and
\begin{eqnarray} \label{gammaprimo}
\qquad\tilde{\Gamma}(m)=\left\{\gamma\in\Gamma(m_2)[2]: \begin{array}{l}
\text{if }\gamma\subset\Q^+\text{ then }v_2(\delta(\gamma))\le v_2(m);\\\text{if }\gamma\nsubseteq\Q^+\text{ then }v_2(\delta(\gamma))=v_2(m)+1\end{array}\right\}.
 \end{eqnarray}

 Note that in the product in (\ref{Bgamma}), the position $\ell\nmid 2m$ is equivalent
 to $\ell\nmid m$. In fact, when $m$ is odd, then necessarily, for $\gamma\in\tilde\Gamma(m)$, $\delta(\gamma)$ is also odd.

\section{The case $\Gamma=\langle -1,a\rangle$ with $a\in\Q^+\setminus\{0,1\}$}\label{1a}

In this section we consider the special case 
when $\Gamma=\langle -1,a\rangle$ with $a\in\Q^+\setminus\{0,1,-1\}$.
The rank of $\Gamma$ is $1$ and we write $a=a_0^h$ with $a_0\in\Q^+$  
not a perfect power of a rational number. Further we write $a_0=a_1a_2^2$ where $a_1,a_2\in\Q^+$ 
are uniquely defined by the property that $a_1\in\N$, $a_1>1$ is square free. 
We have the following:
\begin{theoremm}\label{-1a} With the above notation, let 
$A=\prod_{\ell}\left(1-\frac1{\ell^2-\ell}\right)=0.373955813619202288054\ldots$
be the Artin Canstant,
$$
\rho(\langle -1,a\rangle,m)=\frac{(m,h)}{2m^2}
\prod_{\substack{\ell\mid 2m}}
\frac{\ell^2-\ell}{\ell^2-\ell-1}
\prod_{\substack{\ell\mid h\\\ell\nmid 2m}}\frac{\ell^2-2\ell}{\ell^2-\ell-1}\!\!\!\!
\prod_{\substack{\ell\mid 2m\\ v_\ell(m/h)\ge0}}\!\!\!\!\frac{\ell+1}{\ell}\left(1+\tau_{a,m}\prod_{\substack{\ell\mid a_1\\ \ell\nmid 2m}}
\frac{-(\ell,h)}{\ell^2-\ell-(\ell,h)}\right)A$$
where
$$
\tau_{a,m}=\begin{cases}
0&\text{if }v_2(h)>v_2(m),\text{ or}\\
  &\text{if }v_2(h)=v_2(m)=0\text{ and }2\mid ha_1;\\
  \\
-\frac13&\text{if }v_2(h)=v_2(m)=0\text{ and }2\nmid ha_1,\text{ or}\\
            &\text{if }v_2(h)=v_2(m)>0,\text{ or}\\
            &\text{if }v_2(h)<v_2(m)=1\text{ and }2\mid ha_1;\\
 \\
1&\text{if }v_2(h)<v_2(m)=1\text{ and }2\nmid ha_1,\text{ or}\\
  &\text{if }v_2(h)<v_2(m)\neq1.\\
\end{cases}$$
\end{theoremm}

\begin{proof} For $m\in\N$ (see \cite[equation (5) page 6]{PS}) we have that, 
\begin{equation}\label{esplcalcm}
\left|\langle -1,a\rangle(m)\right|=\left|\langle -1,a_0^h\rangle{\Q^*}^m/{\Q^*}^m\right|=\frac{(2,m)m}{(m,h)}.
\end{equation}
Hence $A_{\langle -1,a\rangle,m}$, as in Theorem~\ref{quasi_r_artin_density}, equals 
$$\frac{(m,h)}{(m,2)\varphi(m^2)}\times
\prod_{\substack{\ell\nmid 2hm}}\left(1-\frac1{(\ell-1)\ell}\right)\times
\prod_{\substack{\ell\nmid 2m\\ \ell\mid h}}\left(1-\frac1{\ell-1}\right)\times\!\!\!\!
\prod_{\substack{\ell>2\\ \ell\mid m\\ v_\ell(h/m)\ge1}}\!\!\!\left(1-\frac1\ell\right)\times\!\!\!\!
\prod_{\substack{\ell>2\\ \ell\mid m\\ v_\ell(h/m)\le0}}\!\!\!\left(1-\frac1{\ell^2}\right).$$
We recall that
\begin{equation*}\widetilde{\langle -1,a\rangle}(m)=
\begin{cases}
 \{1\} &\text{if }2\nmid m;\\
 \{\gamma\in\Gamma(2):  \gamma'\equiv1\bmod4\}&\text{if }2\|m;\\
  \{\gamma\in\Gamma(4): \text{either }\gamma'=\gamma_0^2, 
    2\nmid \gamma_0\text{ or }\gamma'=-\gamma_0^2, 2\mid \gamma_0\}
    &\text{if }4\|m;\\
  \Gamma(m_2)[2]\cap\Q^+&\text{if }8\mid m.
\end{cases}
\end{equation*}
Furthermore, if $\alpha\in\N$, then
$${\langle -1,a\rangle}(2^\alpha)[2]=\begin{cases}
\{{\Q^*}^{2^\alpha},-{\Q^*}^{2^\alpha},a_1^{2^{\alpha-1}}{\Q^*}^{2^\alpha},-a_1^{2^{\alpha-1}}{\Q^*}^{2^\alpha}\}&\text{if }v_2(h)<\alpha\\
\{{\Q^*}^{2^\alpha},-{\Q^*}^{2^\alpha}\}&\text{if }v_2(h)\ge\alpha.
                                     \end{cases}
$$
Therefore, if $v_2(m)=1$
$$\widetilde{\langle -1,a\rangle}(m)=\begin{cases}
                               \{{\Q^*}^2\}&\text{if }2\mid ha_1;\\
                               \{{\Q^*}^2,\left(\frac{-1}{a_1}\right)a_1{\Q^*}^2\} &2\nmid ha_1 
                              \end{cases}, $$
if $v_2(m)=2$                              
$$\widetilde{\langle -1,a\rangle}(m)=\begin{cases}
                               \{{\Q^*}^4\}               &\text{if }4\mid h;\\
                               \{{\Q^*}^4,a_1^2{\Q^*}^4\} &\text{if }2\nmid a_1\text{ and }4\nmid h;\\
                               \{{\Q^*}^4,-a_1^2{\Q^*}^4\}&\text{if }2\mid  a_1\text{ and }4\nmid h
                              \end{cases}$$
and if $\alpha=v_2(m)\ge3$, 
$$\widetilde{\langle -1,a\rangle}(m)=\begin{cases}\{{\Q^*}^{2^\alpha}\}&\text{if }v_2(h)\ge v_2(m)\\\{{\Q^*}^{2^\alpha},a_1^{2^{\alpha-1}}{\Q^*}^{2^\alpha}\}&\text{if }v_2(h)<v_2(m).\end{cases}$$

From this, we deduce that
$$B_{\langle -1,a\rangle,m}=\sum_{\gamma\in\widetilde{\langle -1,a\rangle}(m)}\prod_{\substack{ \ell\mid\delta(\gamma)\\ \ell\nmid 2m}}
\frac{-1}{(\ell-1)|\widetilde{\langle -1,a\rangle}(\ell)|-1}
=1+\varepsilon_{m,a}\prod_{\substack{\ell\mid a_1\\ \ell\nmid 2m}}
\frac{-(\ell,h)}{\ell^2-\ell-(\ell,h)}$$
where
$$\varepsilon_{m,a}=\begin{cases}
0&\text{if }v_2(m)\le v_2(h);\\
0&\text{if }2\|m\text{ and }2\mid ha_1;\\
 1&\text{otherwise.}
                                                                        \end{cases}$$
Therefore
$$B_{\langle -1,a\rangle,m}-\frac{\langle -1,a\rangle(m_2)}{(2,m_2)\langle -1,a\rangle(2m_2)}B_{\langle -1,a\rangle,2m}=$$
$$\left(1-\frac{\gcd(h,2m_2)}{4\gcd(h,m_2)}\right)\left(1+\prod_{\substack{\ell\mid a_1\\ \ell\nmid 2m}}
\frac{-(\ell,h)}{\ell^2-\ell-(\ell,h)}\times\frac{\varepsilon_{m,a}-
\frac{\gcd(h,2m_2)}{4\gcd(h,m_2)}\varepsilon_{2m,a}}{1-\frac{\gcd(h,2m_2)}{4\gcd(h,m_2)}}
\right).$$
Finally 
\begin{eqnarray*}
\tau_{m,a}&=&\frac{\varepsilon_{m,a}-
\frac{\gcd(h,2m_2)}{4\gcd(h,m_2)}\varepsilon_{2m,a}}{1-\frac{\gcd(h,2m_2)}{4\gcd(h,m_2)}}=\begin{cases}0&\text{if }v_2(m)<v_2(h);\\
\frac{-\varepsilon_{2m,a}}3&\text{if }v_2(m)=v_2(h)\\
\frac{4\varepsilon_{m,a}-\varepsilon_{2m,a}}{3}&\text{if }v_2(m)>v_2(h);\\
\end{cases}\\
&=&\begin{cases}
0&\text{if }v_2(h)>v_2(m),\text{ or}\\
  &\text{if }v_2(h)=v_2(m)=0\text{ and }2\mid ha_1;\\
  \\
-\frac13&\text{if }v_2(h)=v_2(m)=0\text{ and }2\nmid ha_1,\text{ or}\\
            &\text{if }v_2(h)=v_2(m)>0,\text{ or}\\
            &\text{if }v_2(h)<v_2(m)=1\text{ and }2\mid ha_1;\\
 \\
1&\text{if }v_2(h)<v_2(m)=1\text{ and }2\nmid ha_1,\text{ or}\\
  &\text{if }v_2(h)<v_2(m)\neq1;\\
\end{cases}
\end{eqnarray*}
and this concludes the proof.\end{proof}
                                                                        
\section{The vanishing of $\rho(\langle -1,a\rangle,m)$ and the proof of Theorem~\ref{week_vanish}}\label{-1,a}

In this section we consider the equation:
$$\rho(\langle -1,a\rangle,m)=0.$$
In virtue of Theorem~\ref{-1a} and on the easy to deduce fact that for every $a\in\Q^+\setminus\{0,1\}$ and $m\in\N$,
 in order to have $\rho(\langle -1,a\rangle,m)=0$, we have to consider the identity:
$$C_{a,m}=1+\tau_{a,m}\prod_{\substack{\ell\mid a_1\\ \ell\nmid 2m}}
\frac{-(\ell,h)}{\ell^2-\ell-(\ell,h)}=0.$$
It is easy to check that, for $\ell$ odd, 
$$\frac{(\ell,h)}{\ell^2-\ell-(\ell,h)}\le 1$$
and the equality holds if and only if $\ell=3\mid h$. Hence the equation
$C_{a,m}=0$ is equivalent to:
$\tau_{a,m}=1$, $3\mid h$ and $3$ is the only odd prime that
divides $a_1$ but it does not divide $m$. 
This happens exactly in one the following cases:
$$3\mid h,  3\nmid m, 3\mid  a_1,  a_1\mid3m,\quad
2\nmid h, 2\|m, 2\nmid a_1,$$
or 
$$3\mid h,  3\nmid m, 3\mid  a_1,  a_1\mid3m,\quad  v_2(h)<v_2(m)\ne1.$$

\begin{proof}[Proof of Theorem~\ref{week_vanish}] From the above discussion, it is clear that $\rho_{\langle-1,a\rangle,m}=0$ is satisfied if and only if one the the above properties are satisfied. In all other cases $\rho_{\langle-1,a\rangle,m}  \ne0$. So, on GRH by \cite[Theorem~1]{PS}, there are infinitely primes $p$ such that $\left[\F_p^*:\langle -1,a\rangle_p\right]=m$.

Suppose next that $a, m$ are such 
$$3\mid h,  3\nmid m, 3\mid  a_1,  a_1\mid3m\text{ and }2\mid m$$
and let $p$ be a prime such that $\left[\F_p^*:\langle -1,a\rangle_p\right]=m$.
From the fact that $2\mid\left[\F_p^*:\langle -1,a\rangle_p\right]$ we deduce that $-1$ and $a$ are squares in $\F_p*$ and that $p\equiv1\bmod2m$.
Furthermore, if $\ell>3$ is any other prime that divides $a_1$. Then $\ell\mid m\mid p-1$. So, by quadratic reciprocity,
$$\left(\frac{\ell}p\right)=\left(\frac p{\ell}\right)=\left(\frac1{\ell}\right)=1.$$

If the first of the properties in the statement of Theorem~\ref{week_vanish} is satisfied, then, since $2\nmid h$, also $a_1$ is a square in $\F_p^*$. The property that $2\nmid a_1$ implies that every $\ell\mid a_1$ has $\left(\frac{\ell}p\right)=1$. Thus
$$\left(\frac{-3}p\right)=\left(\frac{-1}p\right)\left(\frac{3}p\right)=\left(\frac{3}p\right)\prod_{\ell\mid a_1,\ell\ne3}\left(\frac{\ell}p\right)=\left(\frac{a_1}p\right)=1.$$
This implies that $p\equiv1\bmod3$. Hence both $-1$ and $a$ are cubes in $\F_p^3$ which implies that
$3\mid m$ and this is a contradiction.

In the case when $a, m$ are such that the second properties in the statement of Theorem~\ref{week_vanish} is satisfied we let $p$ be a prime such that $\left[\F_p^*:\langle-1,a\rangle_p\right]=m$. Then, since $v_2(h)<v_2(m)$ and $m\mid p-1$,
$$\left(\frac{a_1}{p}\right)=\left(\frac{a_0}{p}\right)^{h/h_2}\equiv {a_0^{h/h_2}}^{\frac{p-1}2}={a}^{\frac{p-1}{2^{h_2+1}}}={a^{\frac{p-1}{m}}}^{m/2^{h_2+1}}\equiv1\bmod p.$$
So that again $a_1$ is a square modulo $p$.
Furthermore, since $v_2(m)\ge2$ and $p\equiv1\bmod2m$, then $8\mid p-1$.
Thus
$$\left(\frac{2}{p}\right)=1$$
Finally, a similar argument as above shows that $\left(\frac{-3}p\right)=1$ and $3\mid p-1$. Again both $-1$ and $a$ are cubes in $\F_p^3$ which implies that
$3\mid m$ and this is a contradiction.
\end{proof}

\section{The vanishing of $\rho(\Gamma,m)$}\label{vanishing}

\begin{proof}[Proof of Theorem~\ref{finite}] We start from the identity
$$\rho(\Gamma,m)=A(\Gamma,m)\left(B_{\Gamma,m}-\frac{|\Gamma(m_2)|}{(2,m)|\Gamma(2m_2)|} B_{\Gamma,2m}\right).$$
It is easy to check, by the definition in (\ref{AGamma}), that $A(\Gamma,m)\neq0$ for all $m$ and all $\Gamma$. So, the equation $\rho(\Gamma,m)=0$ is equivalent to
\begin{equation}\label{ro0}
B_{\Gamma,m}=\frac{|\Gamma(m_2)|}{(2,m)|\Gamma(2m_2)|} B_{\Gamma,2m}. 
\end{equation}
\begin{enumerate}
 \item 
If $2\nmid m$, then $B_{\Gamma,m}=1$ and $|\Gamma(m_2)|=1$.  So the identity in (\ref{ro0}) specializes to
\begin{equation}\label{above}
 |\Gamma(2)|=B_{\Gamma,2m}=\sum_{\gamma\in \tilde\Gamma(2m)}\prod_{\substack{\ell\mid \delta(\gamma)\\
\ell\nmid2m}}\frac{-1}{(\ell-1)|\Gamma(\ell)|-1}.
\end{equation}
Note that the hypothesis that $\disc(\Q(\sqrt g))\mid m$ for all $g\in\Gamma$, we deduce that $\disc(\Q(\sqrt g ))=\delta(g{\Q^*}^2)\equiv 1\bmod4$. 
Hence each of the products in (\ref{above}) is empty. 
Finally $\Gamma(2)=\tilde\Gamma(2m)$ so that the identity in (\ref{above}) is satisfied.

\item Next assume that the condition in \emph{2.} is satisfied. We claim that $B_{\Gamma,m}=B_{\Gamma,2m}=0$ which
implies that (\ref{ro0}) is an identity. Observe that, if $\gamma_1\in\tilde\Gamma(m)$ is
as in the statement, then
$$\prod_{\substack{\ell\mid \delta(\gamma_1)\\
\ell\nmid2m}}\frac{-1}{(\ell-1)|\Gamma(\ell)|-1}=\frac{-1}{2|\Gamma(3)|-1}=-1.$$
Therefore, since $\tilde\Gamma(m)$ is a group, $3\nmid m$ and $3\mid\delta(\gamma_1\gamma)$ if and only of $3\nmid\delta(\gamma)$,
$$B_{\Gamma,m}=-\sum_{\gamma\in\tilde\Gamma(m)}\prod_{\substack{\ell\mid \delta(\gamma_1\gamma)\\
\ell\nmid2m}}\frac{-1}{(\ell-1)|\Gamma(\ell)|-1}=-B_{\Gamma,m}$$
which immediately implies that $B_{\Gamma,m}=0$.
We observe that, if $\gamma_1=\pm\gamma_0^{m_2/2}{\Q^*}^{m_2}$, then
$\gamma_2=\gamma_0^{m_2}{\Q^*}^{2m_2}\in \tilde\Gamma(2m_2)$ since it satisfies $\delta(\gamma_1)=\delta(\gamma_2)$ and $v_2(\gamma_2)\le v_2(2m)$. So, by the same argument, we deduce that $B_{\Gamma,2m}=0$.

\item By the remark after the statement Theorem~\ref{finite}, the third condition implies that $B_{\Gamma,m}=1$ and $|\Gamma(m_2)|=2$. So, identity (\ref{ro0}) reduces to $B_{\Gamma,2m}=|\Gamma(2m_2)|$. The hypothesis that $\Gamma(4)=\tilde\Gamma(2m)$ and that, 
for every $\gamma\in\tilde\Gamma(2m)$,  $\delta(\gamma)\mid 4m$, implies that 
\begin{equation*}
 \prod_{\substack{\ell\mid \delta(\gamma)\\
\ell\nmid2m}}\frac{-1}{(\ell-1)|\Gamma(\ell)|-1}=1
\end{equation*}
so that $B_{\Gamma,2m}=|\tilde\Gamma(2m)|$ and identity (\ref{ro0}) is satisfied. 
\end{enumerate}

\end{proof}

\begin{proof}[Proof of Proposition~\ref{lenstraok}] If $\Gamma=\langle g\rangle$, 
then $3\mid h$ if and only if $\Gamma(3)$ is trivial and that $v_2(h)$ is the largest $\alpha$
such that $\Gamma(2^\alpha)$ is trivial.

To analyze precisely the special case when $\Gamma=\langle g\rangle$, $g=\pm g_0^h$ with $g_0\ne1$ not the power of a rational number, we observe that  $\#\Gamma(m)=m/\gcd(m,h)$ and
 $$\Gamma(m)[2]=
 \begin{cases}
 \{{\Q^*}^{m_2},g_0^{m_2/2}{\Q^*}^{m_2}\} &\text{if }g>0\text{ and }v_2(m)>v_2(h),\text{ or}\\
 &\text{if }g<0\text{ and }v_2(m)>v_2(h)+1;\\
 \{{\Q^*}^{m_2},-g_0^{m_2/2}{\Q^*}^{m_2}\} &\text{if }g<0\text{ and }v_2(m)= v_2(h)+1;\\
  \{{\Q^*}^{m_2},-{\Q^*}^{m_2}\} &\text{if }g<0\text{ and }v_2(m)= v_2(h);\\
 \{{\Q^*}^{m_2}\} &\text{if }g>0\text{ and }v_2(m)= v_2(h),\text{ or}\\
 &\text{if }v_2(m)<v_2(h).
\end{cases}$$

\begin{enumerate}
 \item[A.] If $2\nmid m$ and for all $\gamma\in\Gamma, \disc(\Q(\sqrt\gamma))\mid m$, then, in particular  $\disc(\Q(\sqrt{g})\mid m$ which is the first property in Lenstra's Theorem.
 \item[B.]
If $3\mid\delta(g)\mid6m$, $v_2(\delta(g))\le v_2(m)+1$.
Thus 
 \begin{equation}\label{special}\tilde\Gamma(m)=\begin{cases}
                    \{{\Q^*}^{m_2},g_0^{m_2/2}{\Q^*}^{m_2}\}&\text{if }g>0, v_2(\delta(g))\le v_2(m)\text{ and }v_2(m)>v_2(h),\text{ or}\\
 &\text{if }g<0, v_2(\delta(g))\le v_2(m)\text{ and }v_2(m)>v_2(h)+1;\\
                    \{{\Q^*}^{m_2},-g_0^{m_2/2}{\Q^*}^{m_2}\}&\text{if }g<0\text{ and }v_2(\delta(g))-1=v_2(m)= v_2(h)+1;\\
                    \{{\Q^*}^{m_2}\}&\text{otherwise.}
                   \end{cases}
\end{equation}

Note that, in order for $v_2(\delta(g))-1=v_2(m)$, necessarely $v_2(m)=1$ of $v_2(m)=2$
and in the latter case $2\mid g_0$. The condition  $3\mid\delta(g)\mid6m$  which implies: 
\begin{eqnarray*}
\operatorname{disc}(\Q(\sqrt{-3g_0}))\mid m &&\text{in the first case of  (\ref{special});}\\
\operatorname{disc}(\Q(\sqrt{3g_0}))\mid m &&\text{in the second case of  (\ref{special}), with }v_2(m)=1;\\
\operatorname{disc}(\Q(\sqrt{-6g_0}))\mid m&&\text{in the second case of  (\ref{special}), with }v_2(m)=2.
\end{eqnarray*}
We conclude that the second case of Theorem~\ref{finite} specializes, in the case $\Gamma=\langle g\rangle$, to following cases of the Theorem of Lenstra.
\begin{center}
\begin{tabular}{|l|l|}
\hline
2.& if $g>0$\\
3.& if $g<0$, $v_2(m)=1$ and $v_2(h)=0$ \\
5.& if $g<0$, $v_2(m)=2$ and $v_2(h)=1$\\
6.& if $g<0$ and $v_2(m)>v_2(h)+1$.\\
\hline
\end{tabular}
\end{center}
               \item[C.] The third property in the above statement means that. every element  $\gamma\in\tilde{\Gamma}(4)$ is either of the form $\gamma_0^2{\Q^*}^4$ or $-4\gamma_0^2{\Q^*}^4$ with $\gamma_0\mid m$ odd and square free and at least one of them is of the second form. Hence, necessarily, $g=-g_0^2$ with $g_0$ even, not a fourth power and $v_2(g_0)$ odd. This implies that $2\|h$ and that $\disc(\Q(\sqrt{2g_0}))\mid 2m$.
 \end{enumerate}\end{proof}

\begin{proof}[Proof of Proposition~\ref{oddsuff}]
 Assume that $2\nmid m$ and $\rho(\Gamma,m)=0$, then by (\ref{ro0}), 
$|\Gamma(2)|=B_{\Gamma,2m}$. Furthermore
$$|B_{\Gamma,2m}|\le |\tilde\Gamma(2m)|\le |{\Gamma}(2)|.$$ 
This implies that 
$\tilde{\Gamma}(2m)=\Gamma(2)$ and that for every 
$\gamma\in\Gamma$, $\gamma'\equiv1\bmod4$ and 
\begin{equation*}
 \prod_{\substack{\ell\mid \delta(\gamma)\\
\ell\nmid2m}}\frac{-1}{(\ell-1)|\Gamma(\ell)|-1}=1.
\end{equation*}
Thus $\delta(\gamma)\mid m$ for all $\gamma\in\Gamma(2)$. Hence the property in
\emph{1.} holds for $\Gamma$ and $m$.
\end{proof}

\begin{proof}[Proof of Proposition~\ref{ultima}] Suppose that $\Gamma$ and $m$ satisfy the first condition in the statement of Theorem~\ref{finite}.  
Let $p\not\in\Supp\Gamma$ be such that 
$|\Gamma_p|=(p-1)/m$, then $p\equiv1\bmod m$ and by, quadratic reciprocity, for all $g\in\Gamma$, since $\delta(g{\Q^*}^2)\mid m$, $\left(\frac{g}p\right)=1$. Hence $\Gamma_p\subset\F_p^*$ is contained in the subgroup of squares which implies that $2\mid m$, a contradiction.

Next suppose that $\Gamma$ and $m$ satisfy the second condition in the statement of Theorem~\ref{finite}. 
First note that, if $p\not\in\Supp\Gamma$ is a prime such that $|\Gamma_p|=(p-1)/m$, then $p\equiv 2\bmod3$ since $3\nmid m$ and 
since all elements of $\Gamma$ are perfect cubes. Furthermore the hypothesis $m$ even implies that
all elements of $\Gamma_p$ are squares modulo $p$.
Let $\gamma_1\in\tilde\Gamma(m)$ be such that  $3\mid\delta(\gamma_1)\mid 6m$.
Then
$$\left(\frac{\gamma_1}p\right)=\left(\frac{\delta(\gamma_1)}p\right)=\left(\frac{3}p\right)\left(\frac{\delta(\gamma_1)/3}p\right)=
-1,$$
which is a contradiction to the property that all the elements of $\Gamma$ are squares modulo $p$.

Finally suppose that $\Gamma$ and $m$ satisfy the third condition in the statement of Theorem~\ref{finite}. Let $-4\gamma_0^2{\Q^*}^4\in\Gamma(4)$ with $\gamma_0$ odd and square free as in the Remark after the statement of Theorem~\ref{finite}. Since $2\| m$,  $-4\gamma_0^2$ is a square modulo $p$. Hence $p\equiv1\bmod2m$.  We have also that $p\not\equiv1\bmod4m$, otherwise the quartic symbol 
$$\left[\frac{-4\gamma_0^2}{p}\right]_4=\left[\frac{-1}{p}\right]_4\left(\frac{2}{p}\right)\left(\frac{\gamma_0}{p}\right)=1,$$ 
since $\gamma_0\mid m$. Furthermore  $\gamma_0\mid m$ also implies, by quadratic reciprocity that $\left(\frac{\gamma_0}{p}\right)=1$, hence
the Legendre symbol:
$$\left(\frac{2\gamma_0}{p}\right)=\left(\frac{2}{p}\right)=1$$
if and only if $p\equiv1\bmod8$ (since $p\not\equiv-1\bmod4$). So a contradiction.
%
\end{proof}

\section{Numerical Examples}\label{computations}
In this section we compare numerical data. The density $\rho_{\Gamma,m}$ can be explicitly
computed once a set of generators of $\Gamma$ is given.  The tables in this section have been computed
using Pari-GP \cite{PARI2}.

The first table compares the values of $\rho_{\langle-1,a\rangle,m}$ as in Theorem~\ref{-1a}
(second row) and $$\frac{\pi_{\langle-1,a\rangle}(10899719603,m)}{\pi(10899719603)}\qquad\text{(first row)}$$  with
$a=2,\ldots,21$, $m=1,\ldots,20$.
All values have been truncated to the first decimal digits.

\begin{tiny}
\noindent
\begin{tabular}{c|c|c|c|c|c|c|c|c|c|c}
\hline
 $a\backslash m$ & 1 & 2& 3 & 4 & 5 & 6 & 7 & 8 & 9 & 10\\                                                                                                                         
\hline
2& 0.5609316& 0.09349469& 0.09972896& 0.07011468& 0.02834563& 0.01661614& 0.01340015& 0.01753052& 0.01108010& 0.00472355\\                                          
&0.5609337&0.09348895&0.09972155&0.07011672&0.02834191&0.01662026&0.01340210&0.01752918&0.01108017&0.00472365\\ \hline
3& 0.5983436& 0.1121961& 0.06648385& 0.02804691& 0.03023376& 0.04986213& 0.01429211& 0.007009285& 0.007383818& 0.00566415\\ 
&0.5983293&0.1121867&0.06648103&0.02804669&0.03023138&0.04986078&0.01429557&0.007011672&0.007386782&0.00566838\\ \hline
4& 0.3739585& 0.1869731& 0.06648425& 0.1402365& 0.01889511& 0.03324471& 0.008932948& 0.03505868& 0.007385783& 0.00945051\\
&0.3739558&0.1869779&0.06648103&0.1402334&0.01889461&0.03324052&0.008934733&0.03505836&0.007386782&0.00944730\\ \hline
5& 0.5707797& 0.1328580& 0.1014608& 0.03321178& 0.01889962& 0.02361663& 0.01363818& 0.008306253& 0.01127759& 0.0141702\\
&0.5707747&0.1328527&0.1014711&0.03321318&0.01889461&0.02361826&0.01363722&0.008303295&0.01127456&0.0141709\\ \hline
6& 0.5609309& 0.1495846& 0.09972773& 0.02804226& 0.02834054& 0.01662218& 0.01340140& 0.007010532& 0.01108035& 0.00756130\\
&0.5609337&0.1495823&0.09972155&0.02804669&0.02834191&0.01662026&0.01340210&0.007011672&0.01108017&0.00755784\\ \hline
7& 0.5655185& 0.1368145& 0.1005323& 0.03419843& 0.02856917& 0.02432134& 0.008929491& 0.008552522& 0.01116960& 0.00691573\\
&0.5654942&0.1368131&0.1005323&0.03420328&0.02857234&0.02432233&0.008934733&0.008550819&0.01117025&0.00691266\\ \hline
8& 0.3365588& 0.05609852& 0.2991703& 0.04207116& 0.01700431& 0.04985612& 0.008041882& 0.01051791& 0.03324158& 0.00283249\\
&0.3365602&0.05609337&0.2991647&0.04207003&0.01700515&0.04986078&0.008041260&0.01051751&0.03324052&0.00283419\\ \hline
9& 0.3739683& 0.2991733& 0.06648385& 0.05609534& 0.01889393& 0.03323814& 0.008931910& 0.01402027& 0.007383818& 0.0151180\\
&0.3739558&0.2991647&0.06648103&0.05609337&0.01889461&0.03324052&0.008934733&0.01402334&0.007386782&0.0151156\\ \hline
10& 0.5609298& 0.1427061& 0.09972107& 0.03321470& 0.02834725& 0.02536964& 0.01340418& 0.008301758& 0.01108199& 0.00471766\\
&0.5609337&0.1426937&0.09972155&0.03321318&0.02834191&0.02536776&0.01340210&0.008303295&0.01108017&0.00472365\\ \hline
11& 0.5626496& 0.1389491& 0.1000259& 0.03473188& 0.02843085& 0.02469908& 0.01344676& 0.008686747& 0.01111246& 0.00701644\\
&0.5626491&0.1389469&0.1000265&0.03473672&0.02842859&0.02470167&0.01344308&0.008684180&0.01111406&0.00702047\\ \hline
12& 0.5983387& 0.1121865& 0.06648742& 0.02804858& 0.03023264& 0.04986241& 0.01429060& 0.007011669& 0.007378899& 0.00566779\\
&0.5983293&0.1121867&0.06648103&0.02804669&0.03023138&0.04986078&0.01429557&0.007011672&0.007386782&0.00566838\\ \hline
13& 0.5621469& 0.1393328& 0.09993109& 0.03483086& 0.02840633& 0.02476879& 0.01343573& 0.008701322& 0.01110203& 0.00704395\\
&0.5621400&0.1393287&0.09993601&0.03483217&0.02840286&0.02476955&0.01343092&0.008708044&0.01110400&0.00703976\\ \hline
14& 0.5609384& 0.1413718& 0.09973011& 0.03419959& 0.02833696& 0.02513520& 0.01340095& 0.008548704& 0.01107725& 0.00714109\\
&0.5609337&0.1413735&0.09972155&0.03420328&0.02834191&0.02513307&0.01340210&0.008550819&0.01108017&0.00714308\\ \hline
15& 0.5589555& 0.1417091& 0.1014805& 0.03543326& 0.03024462& 0.02362492& 0.01335049& 0.008858559& 0.01127789& 0.00566780\\
&0.5589655&0.1417096&0.1014711&0.03542739&0.03023138&0.02361826&0.01335507&0.008856848&0.01127456&0.00566838\\ \hline
16& 0.3739585& 0.1869731& 0.06648425& 0.09348516& 0.01889511& 0.03324471& 0.008932948& 0.07012322& 0.007385783& 0.00945051\\
&0.3739558&0.1869779&0.06648103&0.09348895&0.01889461&0.03324052&0.008934733&0.07011672&0.007386782&0.00944730\\ \hline
17& 0.5616273& 0.1397238& 0.09985219& 0.03493080& 0.02838022& 0.02484125& 0.01341405& 0.008729947& 0.01109205& 0.00705916\\
&0.5616237&0.1397160&0.09984421&0.03492899&0.02837678&0.02483839&0.01341858&0.008732248&0.01109380&0.00705933\\ \hline
18& 0.5609340& 0.09348952& 0.09972808& 0.07011901& 0.02834935& 0.01661992& 0.01340618& 0.01753497& 0.01108209& 0.00472335\\
&0.5609337&0.09348895&0.09972155&0.07011672&0.02834191&0.01662026&0.01340210&0.01752918&0.01108017&0.00472365\\ \hline
19& 0.5614939& 0.1398239& 0.09982117& 0.03495974& 0.02836823& 0.02486121& 0.01341314& 0.008741045& 0.01108672& 0.00706348\\
&0.5614820&0.1398222&0.09981903&0.03495555&0.02836962&0.02485728&0.01341520&0.008738887&0.01109100&0.00706470\\ \hline
20& 0.5707806& 0.1328471& 0.1014829& 0.03322058& 0.01889355& 0.02362051& 0.01363731& 0.008299904& 0.01127612& 0.0141725\\
&0.5707747&0.1328527&0.1014711&0.03321318&0.01889461&0.02361826&0.01363722&0.008303295&0.01127456&0.0141709\\ \hline
21& 0.5600268& 0.1409286& 0.1005350& 0.03522960& 0.02829520& 0.02431948& 0.01429444& 0.008803046& 0.01116625& 0.00711351\\
&0.5600216&0.1409175&0.1005323&0.03522937&0.02829583&0.02432233&0.01429557&0.008807343&0.01117025&0.00712004\\
\hline\end{tabular}

\noindent
\begin{tabular}{c|c|c|c|c|c|c|c|c|c|c}
\hline
$a\backslash m$ & 11 & 12& 13 & 14 & 15 & 16 & 17 & 18 & 19 & 20\\ \hline                                                                                                         
2& 0.00510744& 0.0124679& 0.00359997& 0.00222978& 0.00503516& 0.00438153& 0.00206209& 0.00184754& 0.00164031& 0.00354257\\
& 0.00510365& 0.0124652& 0.00359751& 0.00223368& 0.00503856& 0.00438229& 0.00206270& 0.00184670& 0.00164041& 0.00354274\\ \hline
3& 0.00544021& 0.0124616& 0.00383648& 0.00268420& 0.00336011& 0.00175025& 0.00220027& 0.00553765& 0.00174867& 0.00141870\\
& 0.00544389& 0.0124652& 0.00383735& 0.00268042& 0.00335904& 0.00175292& 0.00220022& 0.00554009& 0.00174977& 0.00141710\\ \hline
4& 0.00340434& 0.0249248& 0.00239988& 0.00446720& 0.00335940& 0.00876508& 0.00137414& 0.00369432& 0.00109360& 0.00708846\\
& 0.00340243& 0.0249304& 0.00239834& 0.00446737& 0.00335904& 0.00876459& 0.00137514& 0.00369339& 0.00109361& 0.00708548\\ \hline
5& 0.00519359& 0.00590447& 0.00366198& 0.00317556& 0.00336173& 0.00207477& 0.00210207& 0.00262178& 0.00166963& 0.00354352\\
& 0.00519319& 0.00590457& 0.00366063& 0.00317418& 0.00335904& 0.00207582& 0.00209889& 0.00262425& 0.00166919& 0.00354274\\ \hline
6& 0.00510627& 0.0124629& 0.00359992& 0.00357324& 0.00504114& 0.00175411& 0.00206171& 0.00184493& 0.00163881& 0.00141912\\
& 0.00510365& 0.0124652& 0.00359751& 0.00357389& 0.00503856& 0.00175292& 0.00206270& 0.00184670& 0.00164041& 0.00141710\\ \hline
7& 0.00513631& 0.00607994& 0.00362347& 0.00669896& 0.00507949& 0.00213845& 0.00207716& 0.00269970& 0.00165395& 0.00172480\\
& 0.00514514& 0.00608058& 0.00362676& 0.00670105& 0.00507953& 0.00213770& 0.00207947& 0.00270248& 0.00165375& 0.00172817\\ \hline
8& 0.00306486& 0.0373951& 0.00216091& 0.00133867& 0.0151169& 0.00262701& 0.00123701& 0.00554338& 0.000983628& 0.00212638\\
& 0.00306219& 0.0373956& 0.00215851& 0.00134021& 0.0151157& 0.00262938& 0.00123762& 0.00554009& 0.000984246& 0.00212564\\ \hline
9& 0.00340005& 0.0249340& 0.00239661& 0.00715018& 0.00336011& 0.00350570& 0.00137538& 0.00369255& 0.00109416& 0.00283125\\
& 0.00340243& 0.0249304& 0.00239834& 0.00714779& 0.00335904& 0.00350584& 0.00137514& 0.00369339& 0.00109361& 0.00283419\\ \hline
10& 0.00510461& 0.00590093& 0.00360070& 0.00340683& 0.00503538& 0.00207639& 0.00206380& 0.00281799& 0.00164099& 0.00354335\\
& 0.00510365& 0.00590457& 0.00359751& 0.00340931& 0.00503856& 0.00207582& 0.00206270& 0.00281864& 0.00164041& 0.00354274\\ \hline
11& 0.00340353& 0.00617265& 0.00361100& 0.00331679& 0.00505721& 0.00216874& 0.00206702& 0.00274380& 0.00164561& 0.00175059\\
& 0.00340243& 0.00617542& 0.00360852& 0.00331979& 0.00505397& 0.00217105& 0.00206901& 0.00274463& 0.00164543& 0.00175512\\ \hline
12& 0.00544441& 0.0124649& 0.00383767& 0.00268083& 0.00336152& 0.00175182& 0.00219725& 0.00554242& 0.00175120& 0.00141594\\
& 0.00544389& 0.0124652& 0.00383735& 0.00268042& 0.00335904& 0.00175292& 0.00220022& 0.00554009& 0.00174977& 0.00141710\\ \hline
13& 0.00511588& 0.00619055& 0.00239962& 0.00332478& 0.00504846& 0.00217509& 0.00206879& 0.00274989& 0.00164528& 0.00175839\\
& 0.00511463& 0.00619239& 0.00239834& 0.00332891& 0.00504940& 0.00217701& 0.00206714& 0.00275217& 0.00164394& 0.00175994\\ \hline
14& 0.00510613& 0.00608177& 0.00359234& 0.00223823& 0.00503918& 0.00213949& 0.00206543& 0.00279089& 0.00163908& 0.00172924\\
& 0.00510365& 0.00608058& 0.00359751& 0.00223368& 0.00503856& 0.00213770& 0.00206270& 0.00279256& 0.00164041& 0.00172817\\ \hline
15& 0.00508856& 0.00590030& 0.00358190& 0.00338418& 0.00335933& 0.00221683& 0.00205424& 0.00262339& 0.00162967& 0.00141946\\
& 0.00508574& 0.00590457& 0.00358489& 0.00338579& 0.00335904& 0.00221421& 0.00205547& 0.00262425& 0.00163465& 0.00141710\\ \hline
16& 0.00340434& 0.0166159& 0.00239988& 0.00446720& 0.00335940& 0.0175294& 0.00137414& 0.00369432& 0.00109360& 0.00472694\\
& 0.00340243& 0.0166203& 0.00239834& 0.00446737& 0.00335904& 0.0175292& 0.00137514& 0.00369339& 0.00109361& 0.00472365\\ \hline
17& 0.00510339& 0.00620499& 0.00359956& 0.00333673& 0.00504253& 0.00218122& 0.00137517& 0.00275780& 0.00164143& 0.00176759\\
& 0.00510993& 0.00620960& 0.00360194& 0.00333816& 0.00504476& 0.00218306& 0.00137514& 0.00275982& 0.00164243& 0.00176483\\ \hline
18& 0.00510607& 0.0124616& 0.00359556& 0.00223293& 0.00504174& 0.00438445& 0.00206104& 0.00184680& 0.00163940& 0.00353899\\
& 0.00510365& 0.0124652& 0.00359751& 0.00223368& 0.00503856& 0.00438229& 0.00206270& 0.00184670& 0.00164041& 0.00354274\\ \hline
19& 0.00510559& 0.00621063& 0.00360779& 0.00333827& 0.00504638& 0.00218332& 0.00206461& 0.00276301& 0.00109126& 0.00176729\\
& 0.00510864& 0.00621432& 0.00360103& 0.00334070& 0.00504349& 0.00218472& 0.00206472& 0.00276192& 0.00109361& 0.00176618\\ \hline
20& 0.00519737& 0.00589984& 0.00366014& 0.00317452& 0.00336146& 0.00207540& 0.00209665& 0.00262335& 0.00166906& 0.00353976\\
& 0.00519319& 0.00590457& 0.00366063& 0.00317418& 0.00335904& 0.00207582& 0.00209889& 0.00262425& 0.00166919& 0.00354274\\ \hline
21& 0.00509391& 0.00607751& 0.00358806& 0.00268179& 0.00508293& 0.00220359& 0.00206029& 0.00270407& 0.00163689& 0.00178045\\
& 0.00509535& 0.00608058& 0.00359166& 0.00268042& 0.00507953& 0.00220184& 0.00205935& 0.00270248& 0.00163774& 0.00178001\\ \hline
\end{tabular}
\bigskip\end{tiny}

Next table lists the first few values of $a$ (first raw), its factorization (second raw) and $m$ (third raw) such that $\rho(\langle-1,a\rangle,m)=0$.\medskip

\begin{tiny}
\noindent\begin{tabular}{l|l|l|l|l|l|l|l|l|l|l|l|l|l|l|l|l|l}
\hline
  $a$\!&        27\!&    216\!&   729&  1728&    3375&     9261&    13824& 19683&  27000&  35937& 46656&  59319& 74088&110592&132651&185193&216000\\
        &$3^3$&$6^3$& $3^6$&$12^3$&$15^3$&$21^3$& $24^3$&$3^9$&$30^3$&$33^3$&$6^6$&$39^3$&$42^3$&$48^3$&$51^3$&$57^3$&$60^3$\\ 
 $m$\!&$2$&$4$&$4$&2&10&14&4&2&20&22&4&26&28&2&34&38&10\\
 \hline
\end{tabular}
\end{tiny}\bigskip

Next table compares the values of $\rho_{\Gamma,m}$ as in Theorem~\ref{quasi_r_artin_density}
(second row) and $$\frac{\pi_{\Gamma}(10^{10},m)}{\pi(10^{10})}\qquad\text{(first row)}$$  for
some groups $\Gamma$ of rank $2$ and $m=1,\ldots,20$.
All values have been truncated to the first decimal digits.\medskip

\noindent\begin{tiny}
\!\!\!\!\!\!\! \begin{tabular}{c|c|c|c|c|c|c|c|c|c|c}
\hline
$\Gamma\backslash m$& 1 & 2& 3 & 4 & 5 & 6 & 7 & 8 & 9 & 10\\ \hline                                                                                                              
$\langle-1,2,3\rangle$
                                 & 0.820596& 0.082060& 0.0395175& 0.0239324& 0.00822387& 0.0098772& 0.00279091& 0.0029907& 0.0014603& 0.0008217\\
                                 & 0.820590& 0.082059& 0.0395099& 0.0239339& 0.00822248& 0.0098774& 0.00279248& 0.0029917& 0.0014633& 0.0008222\\ \hline
$\langle2,3\rangle$  & 0.697505& 0.205153& 0.0395175& 0.0205123& 0.00698931& 0.0098772& 0.00237151& 0.0059838& 0.0014603& 0.0020563\\
                                 & 0.697501& 0.205147& 0.0395099& 0.0205147& 0.00698910& 0.0098774& 0.00237361& 0.0059834& 0.0014633& 0.0020556\\ \hline
$\langle2,-3\rangle$ & 0.711182& 0.191476& 0.0263467& 0.0205125& 0.00712861& 0.0230480& 0.00241891& 0.0059831& 0.0009733& 0.0019170\\
                                 & 0.711178& 0.191471& 0.0263399& 0.0205147& 0.00712615& 0.0230474& 0.00242015& 0.0059834& 0.0009755& 0.0019185\\ \hline
$\langle-2,3\rangle$ & 0.697509& 0.205148& 0.0395175& 0.0205138& 0.00699074& 0.0098772& 0.00237228& 0.0059827& 0.0014603& 0.0020548\\
                                 & 0.697501& 0.205147& 0.0395099& 0.0205147& 0.00698910& 0.0098774& 0.00237361& 0.0059834& 0.0014633& 0.0020556\\ \hline
$\langle-2,-3\rangle$& 0.711187& 0.191471& 0.0263420& 0.0205148& 0.00712694& 0.0230528& 0.00241881& 0.0059807& 0.0009757& 0.0019186\\
                                 & 0.711178& 0.191471& 0.0263399& 0.0205147& 0.00712615& 0.0230474& 0.00242015& 0.0059834& 0.0009755& 0.0019185\\ \hline
\hline
$\Gamma\backslash m$ & 11 & 12& 13 & 14 & 15 & 16 & 17 & 18 & 19 & 20\\ \hline                                                                                       
$\langle-1,2,3\rangle$
                                & 0.000679& 0.002879& 0.0004043& 0.0002789& 0.000396198& 0.000373124& 0.000176883& 0.000364441& 0.000126355& 0.000239328\\
                                & 0.000678& 0.002880& 0.0004046& 0.0002792& 0.000395897& 0.000373967& 0.000177465& 0.000365832& 0.000126284& 0.000239822\\ \hline
$\langle2,3\rangle$ & 0.000577& 0.002468& 0.0003435& 0.0006983& 0.000396198& 0.000747096& 0.000150315& 0.000364441& 0.000107638& 0.000205653\\
                                & 0.000576& 0.002469& 0.0003439& 0.0006981& 0.000395897& 0.000747933& 0.000150846& 0.000365832& 0.000107342& 0.000205562\\ \hline
$\langle2,-3\rangle$& 0.000588& 0.002469& 0.0003505& 0.0006509& 0.000263545& 0.000747221& 0.000153294& 0.000851385& 0.000109579& 0.000205082\\
                                & 0.000587& 0.002469& 0.0003506& 0.0006515& 0.000263931& 0.000747933& 0.000153803& 0.000853609& 0.000109447& 0.000205562\\ \hline
$\langle-2,3\rangle$& 0.000576& 0.002469& 0.0003436& 0.0006975& 0.000396198& 0.000746852& 0.000150279& 0.000364441& 0.000107482& 0.000205266\\
                                & 0.000576& 0.002469& 0.0003439& 0.0006981& 0.000395897& 0.000747933& 0.000150846& 0.000365832& 0.000107342& 0.000205562\\ \hline
$\langle-2,-3\rangle$& 0.000588& 0.002467& 0.0003507& 0.0006510& 0.000263912& 0.000747661& 0.000153299& 0.000848999& 0.000109390& 0.000204851\\
                                 & 0.000587& 0.002469& 0.0003506& 0.0006515& 0.000263931& 0.000747933& 0.000153803& 0.000853609& 0.000109447& 0.000205562\\ \hline
\end{tabular}
\end{tiny}

\section{Conclusion}
The main question that remains open is whether the sufficient conditions of Theorem~\ref{finite} for $\rho(\Gamma,m)=0$ 
are also necessary. Furthermore, in the present paper we did not address the problem of how to efficiently compute $\rho(\Gamma,m)$. This will be the topic of a coming paper in \cite{MuPa}. Finally, in another paper \cite{MuPa2}, we shall propose
formulas for the density of the primes $p$ for which $m\mid \#\Gamma_p$ where $m$ and $\Gamma$ satisfy the same properties of the present paper.\bigskip

\noindent\textsc{Acknowledgments.} The authors were supported in part by GNSAGA from INDAM. The second author was also supported by PRIN 2017JTLHJR  \textit{Geometric, algebraic and analytic methods in arithmetic}.

\end{document}